\documentclass[a4paper,10pt]{article}

\usepackage[utf8]{inputenc} 
\usepackage{amssymb,latexsym}
\usepackage{amsmath,amsthm}
\usepackage{graphicx}
\usepackage{color,epsfig}
\usepackage{physics}
\usepackage{hyperref}
\usepackage{cleveref}
\usepackage{enumerate}     
\usepackage{xcolor}
\usepackage{colortbl}
\usepackage{blindtext}
\usepackage{array}
\usepackage{tikz}
\usepackage{enumitem}
\usepackage{mathrsfs}
\usepackage{wrapfig}

\usepackage{mathtools}
\input epsf
\usepackage{epstopdf}
\usepackage{subcaption}
\usepackage[overload]{empheq}
\usepackage{multicol}
\usepackage{currvita}
\usepackage{float}
\usepackage[justification=centering]{caption}
\usepackage{changepage}
\usepackage{relsize}
\usepackage{svg}
\usepackage{etoolbox,refcount}
\usepackage{titling}

\newtheorem{theorem}{Theorem}
\newtheorem{lemma}{Lemma}

\newtheorem{problem}{Problem}

\newtheorem{remark}{Remark}
\newtheorem{claim}{Claim}

\newtheorem{conjecture}{Conjecture}

\def\cro{{\mbox {\sc cr}}}

\newcounter{countitems}
\newcounter{nextitemizecount}
\newcommand{\setupcountitems}{%
  \stepcounter{nextitemizecount}%
  \setcounter{countitems}{0}%
  \preto\item{\stepcounter{countitems}}%
}
\makeatletter
\newcommand{\computecountitems}{%
  \edef\@currentlabel{\number\c@countitems}%
  \label{countitems@\number\numexpr\value{nextitemizecount}-1\relax}%
}
\newcommand{\nextitemizecount}{%
  \getrefnumber{countitems@\number\c@nextitemizecount}%
}
\newcommand{\previtemizecount}{%
  \getrefnumber{countitems@\number\numexpr\value{nextitemizecount}-1\relax}%
}
\makeatother    
{\end{itemize}%
\unskip\computecountitems\ifnumcomp{\previtemizecount}{>}{3}{\end{multicols}}{}}

\title{1-planar unit distance graphs}

\thanksmarkseries{alph}
\author{
Panna Geh\'er\thanks{E\"otv\"os Lor\'and University, Budapest, Hungary, and Alfr\'ed R\'enyi Institute of Mathematics, Budapest, Hungary. \newline Email:~\href{mailto:geher.panna@ttk.elte.hu}{\tt geher.panna@ttk.elte.hu}.}
\and
G\'eza T\'oth\thanks{Alfr\'ed R\'enyi Institute of Mathematics, Budapest, Hungary. Email:~\href{mailto:geza@renyi.hu}{\tt geza@renyi.hu}.}
}

\begin{document}

\maketitle

%%%%%%%%%%%%%%%%%%%%%%%%%%%%%%%%%%%%%%%%%%%%%%%%%%%%%%%%%%%%%%%%%%%%%
\begin{abstract}
A matchstick graph is a plane graph with edges drawn as unit distance line segments. This class of graphs was introduced by Harborth who conjectured that a matchstick graph on $n$ vertices can have at most $\lfloor 3n - \sqrt{12n - 3}\rfloor$ edges. Recently, his conjecture was settled by Lavoll\'ee and Swanepoel. In this paper we consider $1$-planar unit distance graphs. We say that a graph is a $1$-planar unit distance graph if it can be drawn in the plane such that all edges are drawn as unit distance line segments while each of them are involved in at most one crossing. We show that such graphs on $n$ vertices can have at most $3n-\sqrt[4]{n}/15$ edges, which is almost tight. We also investigate some generalizations, namely $k$-planar and $k$-quasiplanar unit distance graphs.
\end{abstract}


%%%%%%%%%%%%%%%%%%%%%%%%%%%%%%%%%%%%%%%%%%%%%%%%%%%%%%%%%%%%%%%%%%%%%
\section{Introduction}

\smallskip

A graph is called a matchstick graph if it can be drawn in the plane with no crossings such that all edges are drawn as unit segments. This graph class was introduced by Harborth in 1981 \cite{Harborth1, Harborth2}. He conjectured that the maximum number of edges of a matchstick graph with $n$ vertices is $\lfloor 3n - \sqrt{12n - 3}\rfloor$. He managed to prove it in a special case where the unit distance is also the smallest distance among the points \cite{Harborth74}. Recently, his conjecture was settled by Lavoll\'ee and Swanepoel \cite{LS22}.

For any $k\ge 0$, a graph $G$ is called $k$-planar if $G$ can be drawn in the plane such that each edge is invol\-ved in at most $k$ crossings. Let $e_k(n)$ denote the maximum number of edges of a $k$-planar graph on $n$ vertices. Since $0$-planar graphs are the well-known planar graphs, $e_0(n)=3n-6$ for $n\ge 3$. We have $e_1(n)=4n-8$ for $n\ge 4$ \cite{PT97}, $e_2(n)\le 5n-10$, which is tight for infinitely many values of $n$  \cite{PT97}, $e_3(n)\le 5.5n-11$, which is tight up to an additive constant  \cite{PRTT06} and $e_4(n)\le 6n-12$, which is also tight up to an additive constant  \cite{A19}. For general $k$, we have $e_k(n)\le c\sqrt{k}n$ for some constant $c$, 
which is tight apart from the value of $c$ \cite{PT97, A19}.

A $k$-planar unit distance graph is a graph that can be drawn in the plane such that each edge is a 
unit segment and involved in at most $k$ crossings. Let $u_k(n)$ be the maximum number of edges 
of a $k$-planar unit distance graph on $n$ vertices. Since $0$-planar unit distance graphs are 
exactly the matchstick graphs, by the result of Lavoll\'ee and Swanepoel, 
we have $u_0(n)=\lfloor 3n - \sqrt{12n - 3}\rfloor$. 
Clearly, $u_1(n)\ge u_0(n)$, and for most of the values of $n$, 
we do not have any 
better lower bound for $u_1(n)$ than the value of $u_0(n)$. 
That is, allowing to use one crossing on each edge does not seem to help, still a proper piece of the triangular grid is the best known construction. However, very recently
\v{C}ervenkov\'a \cite{C25} found a construction that has $u_0(n)+1$ or $u_0(n)+2$ edges, for infinitely
many values of $n$. 
Somewhat surprisingly, we prove an almost matching upper bound.

\begin{theorem}\label{1-planar}
For the maximum number of edges of a $1$-planar unit distance graph $u_1(n)$, we have
$$\lfloor 3n - \sqrt{12n - 3}\rfloor\le u_1(n)\le 3n-\sqrt[4]{n}/15.$$
\end{theorem}

For general $k$, the best known lower bound is due to G\"unter Rote (personal communication, 2023).

\begin{theorem}\label{rote} (Rote)
For the maximum number of edges of a $k$-planar
unit distance graph $u_k(n)$, we have
$$u_k(n) \geq 2^{\Omega\left(\log k/\log\log k \right)}n.$$
\end{theorem}

We include the proof in this note. We have the following upper bound.

\begin{theorem}\label{k-planar} For any $n, \, k\ge 0$,
we have
$$u_k(n)\le c\sqrt[4]{k}n$$ for some $c>0$.
\end{theorem}

A graph is called $k$-quasiplanar if it can be drawn in the plane with no $k$ pairwise crossing edges. The following is a long-standing conjecture \cite{BMP05}. 

\begin{conjecture}
For any $k>1$, a $k$-quasiplanar graph on $n$ vertices 
can have at most $c_kn$ edges for some $c_k>0$.
\end{conjecture}

The conjecture has been verified only for $k\le 4$ \cite{AT07, A09}. In general, the best known upper bound is $n(\log n)^{O(\log k)}$ \cite{FP14} (see also \cite{FPS24}) and $O(n\left(\log n)^{4k-16}\right)$ for $k \geq 4$ \cite{PRT06}.

For {\em geometric graphs}, that is, where edges are drawn as straight line segments, the best known upper bound is $O(n\log n)$ \cite{V98}.

A graph is called a $k$-quasiplanar unit distance graph if it can be drawn in the plane with unit segments as edges such that there are no $k$ pairwise crossing edges. Let $v_k(n)$ be the maximum number of edges of such a graph on $n$ vertices. 
For any fixed $k$, a linear upper bound for $v_k(n)$ follows from a result of Suk \cite{S14}. 
Here we prove a much better linear upper bound and a similar lower bound. 

\begin{theorem}\label{quasi}
For the maximum number of edges of a $k$-quasiplanar unit distance graph on $n$ vertices, $v_k(n)$, we have:

\begin{enumerate}[label=(\roman*)] 
\item $v_k(n)< 4kn$,
\item $(k-1)n-o(n)\le v_k(n)$,
whenever $k=2^{O\left(\log n/\log\log n \right)}$.
\end{enumerate}
\end{theorem}

\smallskip

\noindent {\bf Paper outline.}
In \Cref{sec2}, we study $1$-planar unit distance graphs, and 
prove \Cref{1-planar}. In \Cref{sec3}, $k$-planar unit distance graphs are considered, and we prove \Cref{rote} and \Cref{k-planar}. In \Cref{sec4}, we examine $k$-quasiplanar graphs, and we prove \Cref{quasi}. Finally, in \Cref{open_questions}, we state some interesting open problems.

%%%%%%%%%%%%%%%%%%%%%%%%%%%%%%%%%%%%%%%%%%%%%%%%%%%%%%
\section{1-planar unit distance graphs} \label{sec2}
\begin{proof}[\bf Proof of \Cref{1-planar}.]
The lower bound follows directly from Harborth's lower bound for matchstick graphs \cite{Harborth74}. We prove the upper bound. Let $G$ be a $1$-planar unit distance graph with $n$ vertices and consider a $1$-plane unit distance drawing of $G$. Let $E$ be the set of edges, $|E|=e$. Let $G_0$ be a plane subgraph of $G$ with maximum number of edges, and among those one with the minimum number of triangular  faces. Let $E_0\subset E$ denote the set of edges of $G_0$ and $E_1=E\setminus E_0$ denote the set of remaining edges, $|E_0|=e_0$, $|E_1|=e_1$. Let $f$ be the number of faces of $G_0$, including the unbounded face and let $\Phi_1, \, \Phi_2 \, \dots, \, \Phi_f$ be the faces of $G_0$. For any face $\Phi_i$, $|\Phi_i|$ is the number of its bounding edges, counted with multiplicity. That is, if an edge bounds $\Phi_i$ from both sides, then it is counted twice. Due to the maximality of $G_0$, every edge $\alpha\in E_1$ crosses an edge in $E_0$ and connects two vertices that belong to neighboring faces of $G_0$.
Therefore, we can partition every edge $\alpha\in E_1$ into two {\em halfedges} at the unique crossing point on $\alpha$. 
Each halfedge is contained in a face $\Phi$, one of its endpoints is a vertex of $\Phi$ the other endpoint is an interior point of a bounding edge. See \Cref{halfedges}. In the rest of the proof, our main goal is to count the number of halfedges for each face and argue that there are not too many of them.

\begin{figure}[ht]
\centering
\includegraphics[width=6.3cm]{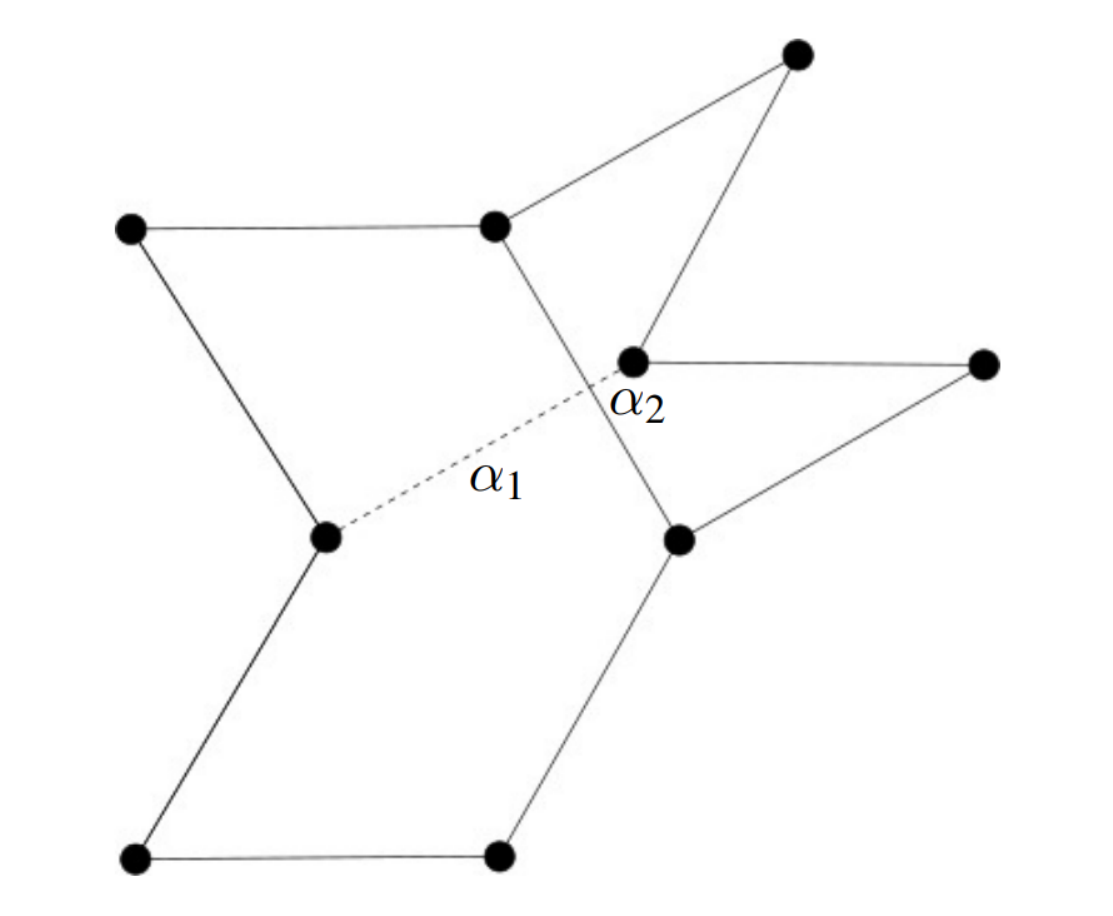}
\caption{An edge in $E_1$ -- drawn with dashed lines -- can be partitioned into two halfedges: $\alpha_1$ and $\alpha_2$.}
\label{halfedges}
\end{figure}

\begin{claim} \label{claim:haromszog}
A triangular face of $G_0$ that does not contain isolated vertices, does not contain any halfedges.
\end{claim}

\begin{proof}
Let $\Phi=uvw$ be a triangular face of $G_0$ with no isolated vertices. It is easy to see that $G_0$ can contain at most one halfedge by 1-planarity. Suppose that $G_0$ contains a halfedge $\alpha_1$ that is part of the edge $\alpha=ux$. Then $\alpha$ crosses the edge $vw$. Replace the edge $vw$ by $\alpha$ in $G_0$ (see \Cref{triangle}). Since $vw$ is the only edge of $G$ that crosses $\alpha$, we obtain another plane subgraph of $G$. It has the same number of edges.

\begin{figure}[htp]
\centering
\includegraphics[width=11cm]{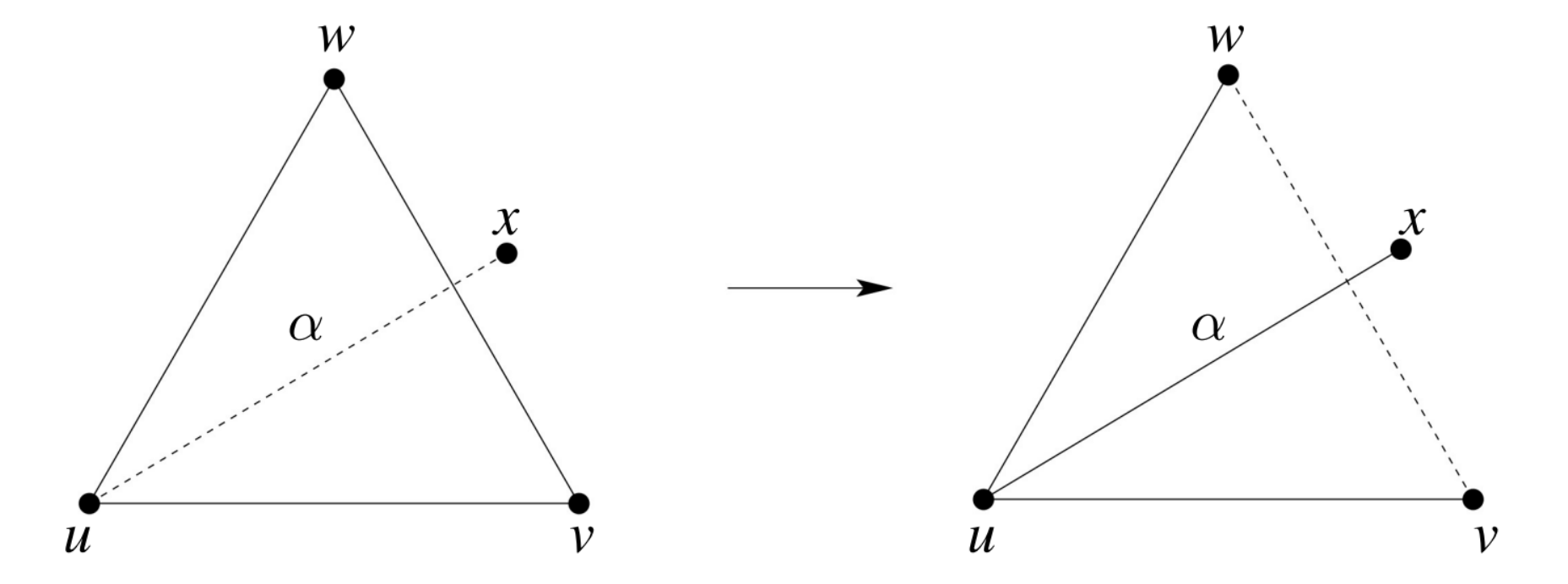}
\caption{If a triangular face with no isolated vertices contains a halfedge, then the number of triangles in $G_0$ can be reduced by an `edge flip'.}
\label{triangle}
\end{figure}

We claim that it has fewer triangular faces. The triangular face $\Phi$ disappeared. Suppose that we have created a new triangular face. 
Then $\alpha$ should be a side of it. Then either $uv$ or $uw$ is also a side, suppose without loss of generality that it is $uv$. But then $uvx$ is also a unit equilateral triangle. If the two equilateral triangles $uvw$ and $uvx$ are on the same side of $uv$, then $x=w$, which is a contradiction. If they are on opposite sides then $vw$ and $ux$ cannot cross each other, which is also a contradiction.
\end{proof}

Assign $1/2$  weight to each halfedge. 
For any face $\Phi$, let $s(\Phi)$ be the sum of the weights of its halfedges. Clearly, we have $$\sum_{i=1}^fs(\Phi_i)=|E_1|.$$
For any face $\Phi$ of $G_0$,  let $t(\Phi)$ denote the number of additional edges needed to triangulate the face $\Phi$. A straightforward consequence of Euler's formula is the following statement. If the boundary of $\Phi$ has $m$ connected components then
\begin{equation} \label{Euler_consequence}
t(\Phi)=|\Phi|+3m-6. 
\end{equation}

\begin{claim} \label{claim:s(f)}
Let $\Phi$ be a face of $G_0$. Then
\begin{enumerate}[label=(\roman*)] 
\item if $|\Phi| < 5$, then we have $s(\Phi)\le t(\Phi)$,
\item if $|\Phi|\ge 5$, then we have
$s(\Phi)\le t(\Phi)-|\Phi|/10.$
\end{enumerate}
\end{claim}

\begin{proof}
Observe that for any face $\Phi$, each of the $|\Phi|$ edges on the boundary of $\Phi$ is crossed by at most one halfedge, therefore,
$s(\Phi)\le |\Phi|/2$. Suppose first that the boundary of $\Phi$ is not connected, that is, $m\ge 2$. By (\ref{Euler_consequence}), $t(\Phi)\ge |\Phi|$, and by the observation above, we have $s(\Phi)\le |\Phi|/2$. Therefore, 
$$t(\Phi)\ge |\Phi|\ge |\Phi|/2+|\Phi|/10 \ge s(\Phi)+|\Phi|/10.$$ 
From now on, we can assume that the boundary of $\Phi$ is connected, that is, $m=1$.

If $|\Phi|=3$ then $\Phi$ is a triangle,
$t(\Phi)=0$ and by \Cref{claim:haromszog} $s(\Phi)=0$, so we are done. If $|\Phi|=4$, then $\Phi$ is a quadrilateral (actually, a rhombus), $t(\Phi)=1$. Suppose that it contains at least $3$ halfedges. 
By the $1$-planarity, they do not intersect each other and they end on different sides of $\Phi$. But then one of the halfedges would end on a side of $\Phi$ which is adjacent to its other endpoint, a vertex of $\Phi$. This is clearly impossible, consequently, 
$s(\Phi)\le 1=t(\Phi)$. A very similar, slightly more detailed argument can be found in \cite{PT97}. For completeness, \Cref{4gons} shows all possible cases when $\Phi$ has two halfedges. This finishes part (i).

\begin{figure}[htp]
\centering
\includegraphics[width=11cm]{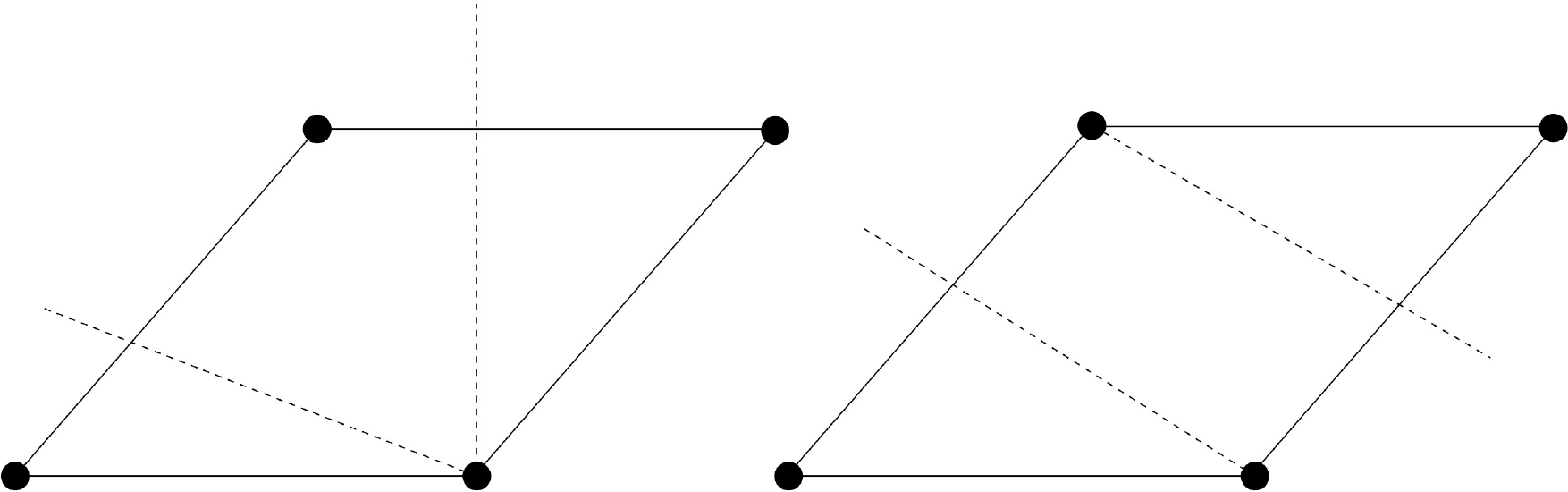}
\caption{A quadrilateral can have at most two halfedges.}
\label{4gons}
\end{figure}

For (ii), let $|\Phi|\ge 5$. We can assume that $\Phi$ has at least two halfedges, otherwise, we are done. 
A halfedge $\alpha$ in $\Phi$ divides $\Phi$ into two parts. Let $a(\alpha)$ and $b(\alpha)$ be the number of vertices of $\Phi$ in the two parts. If a vertex appears on the boundary more than once, then it is counted with multiplicity. Since the halfedges in $\Phi$ do not cross each other, all other halfedges are entirely in one of these two parts. If one part does not contain any halfedges, then $\alpha$ is called a {\em minimal halfedge}. Let $\alpha$ be a 
halfedge for which $M=\min \{ a(\alpha), \, b(\alpha) \}$ is minimal. 
Then there are $M$ vertices of $\Phi$ on one side of $\alpha$. 
Clearly, this part cannot contain any halfedge, so $\alpha$ is minimal. 
Now, for any other halfedge $\beta\neq\alpha$, let $c(\beta)$ be  the number of vertices of $\Phi$ on the side of $\beta$ not containing $\alpha$. 
Take a halfedge $\beta$ for which $c(\beta)$ is minimal. Then $\beta$ is also a minimal halfedge. So, we can conclude that there at least two minimal halfedges in $\Phi$, say, $\alpha$ and $\beta$. 

Then $\alpha$ and $\beta$ together partition $\Phi$ into three parts, two parts contain no other halfedges but both parts contain an edge of $\Phi$. So, at most $|\Phi|-2$ edges of $\Phi$ are crossed by a halfedge, therefore, there are at most $|\Phi|-2$ halfedges in $\Phi$, consequently $s(\Phi)\le (|\Phi|-2)/2$. See \Cref{minimal}.

On the other hand, $t(\Phi)=|\Phi|-3$. Since $|\Phi|\ge 5$, we have
$$t(\Phi)=|\Phi|-3\ge (|\Phi|-2)/2+|\Phi|/10 \ge s(\Phi)+|\Phi|/10.$$ 

This concludes the proof of the Claim. 
\end{proof}

\begin{figure}[t]
\centering
\includegraphics[width=8.5cm]{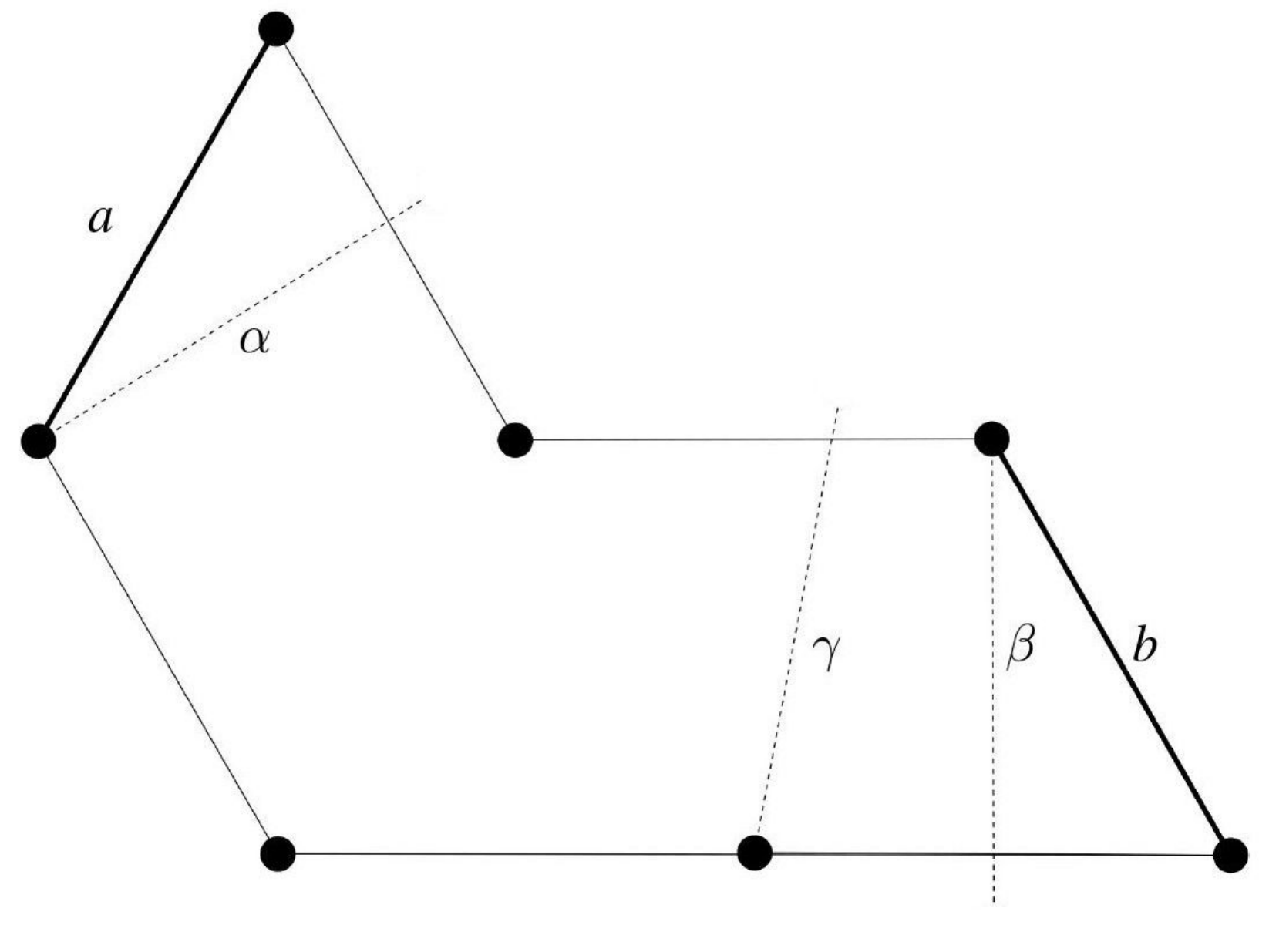}
\caption{Halfedges $\alpha$ and $\beta$ are minimal, edges $a$ and $b$ are uncrossed.}
\label{minimal}
\end{figure}

Return to the proof of \Cref{1-planar}. For $i\ge 3$, let $f_i$ denote the number of faces $\Phi$ of $G_0$ with $|\Phi|=i$. By definition, $\sum_{i=3}^{\infty}f_i=f$ and $\sum_{i=3}^{\infty}if_i=2e_0$. Let $F_{\ge 5}=\sum_{i=5}^{\infty}if_i$. By the maximality of $G_0$, every edge in $E_1$ crosses an edge in $E_0$, and by $1$-planarity, every edge in $E_0$ is crossed by at most one edge in $E_1$. Consequently, $|E_0|=e_0\ge |E_1|=e_1$.

If $e_0\le n$, then $e=e_0+e_1\le 2e_0\le 2n<3n-\sqrt[4]{n}/15$, so we are done. Therefore, for the rest of the proof we can assume that $e_0\ge n$. It follows that 
\begin{align}\label{eq_2}
    3f_3+4f_4+F_{\ge 5}=2e_0\ge 2n.
\end{align}

\begin{claim} \label{claim:f>=5}
Suppose that $F_{\ge 5}\ge p $. Then 
$e=e_0+e_1\le 3n-p/10$.
\end{claim}

\begin{proof}
By the previous observations,

\begin{align*}
e =e_0+e_1=e_0+\sum_{\substack{ \alpha \text{ is a} \\ \text{ halfedge} }} 1/2 &= e_0+\sum_{i=1}^f  s(\Phi_i)\\
&= e_0 + \sum_{|\Phi|=3}s(\Phi) + \sum_{|\Phi|=4}s(\Phi)+ \sum_{|\Phi|\ge 5}s(\Phi)\\
&\le
e_0 + \sum_{|\Phi|=3}t(\Phi) + \sum_{|\Phi|=4}t(\Phi) + \sum_{|\Phi|\ge 5}(t(\Phi)-|\Phi|/10)\\
&\le
e_0 + \sum_{\Phi}t(\Phi) - \sum_{|\Phi|\ge 5}|\Phi|/10\\
%&\stackrel{(*)}{\le}
&\le 3n-6-F_{\ge 5}/10\\
&\leq  3n-p/10.
\end{align*}

We used that a triangulation on $n$ vertices has $3n-6$ edges, so 
$e_0 + \sum_{\Phi}t(\Phi)=3n-6$.

\end{proof}

\begin{claim} \label{claim:f3}
Suppose that $f_{3}\ge p$. Then 
$e=e_0+e_1\le 3n-\sqrt{p}/5$.
\end{claim}

\begin{proof}
We can assume that $\Psi$, the unbounded face of $G_0$ has at least 5 edges. If not, the statement holds trivially.
Since we have $p$ equilateral triangles in $G_0$, the union of all bounded faces, $R$, has area at least $ \sqrt{3}p/4$.
The isoperimetric inequality \cite{C93} states that if a polygon has perimeter $l$ and area $A$, then $l^2\ge 4\pi A$. It implies that $R$ has perimeter at least $\sqrt[4]{3}\sqrt{\pi p}>2\sqrt{p}$. That is $|\Psi|\ge 2\sqrt{p}$. Therefore,
\begin{align*}
    e=e_0+e_1=e_0+ \sum_{\substack{ \alpha \text{ is a} \\ \text{ halfedge} }} 1/2 &=e_0+\sum_{i=1}^f s(\Phi_i)\\
    &= e_0 + \sum_{\Phi\neq \Psi}s(\Phi) + s(\Psi)\\
    &\le e_0 + \sum_{\Phi\neq \Psi}t(\Phi) + t(\Psi) - |\Psi|/10\\
    &= 3n-6 - |\Psi|/10\\
    &\le 3n-6-\sqrt{p}/5.
\end{align*}
\end{proof}

We can assume that $n \geq 5$, otherwise, \Cref{1-planar} holds trivially. If $F_{\ge 5}\ge n/2$, then by \Cref{claim:f>=5}, $e\le 3n-n/20\le 3n-\sqrt[4]{n}/15$ and we are done. If  $f_{3}\ge n/9$, then by \Cref{claim:f3}, 
$e\le 3n-\sqrt{n}/15 \le 3n-\sqrt[4]{n}/15$ and we are done again. So, we can assume that 
$F_{\ge 5}\le n/2$, $f_{3}\le n/9$. Since $3f_3+4f_4+F_{\ge 5} = 2e_0\ge 2n$ (cf.\ \eqref{eq_2}), it follows that $f_4\ge n/4$. 

\smallskip

Suppose without loss of generality that none of the edges of $G$ are vertical. Otherwise, apply a rotation. Define an auxiliary graph $H$ as follows. The vertices represent the quadrilateral faces of $G_0$. Since all edges are of unit length, all these faces are rhombuses. Two vertices are connected by an edge if the corresponding rhombuses share a common edge. The edges of $H$ correspond to the edges of $G_0$ with a rhombus  face on both sides. For every edge in $H$ define its weight as the slope of the corresponding edge of $G_0$. A path in $H$, such that all of its edges have the same weight $w$, is called a {\em $w$-chain}, or briefly a {\em chain}. A chain corresponds to a sequence of rhombuses such that the consecutive pairs share a side and all these sides are parallel. A chain, with at least two vertices (rhombuses) is called {\em maximal} if it cannot be extended. 

With one-vertex chains we have to be more careful. Suppose that $v$ is a vertex of $H$, $R$ is the corresponding rhombus, and let $w_1$, $w_2$ be the slopes of its sides. The one-vertex chain $v$ is a {\em maximal $w_1$-chain} 
(resp.\ {\em maximal $w_1$-chain}) 
if it cannot be extended to a larger $w_1$-chain (resp.\ larger $w_2$-chain). 
Each vertex of $H$, that is, each rhombus face in $G_0$, is in exactly two maximal chains.
Now we prove an important property of rhombus chains. 
Similar ideas were used in \cite{KS92} and \cite{K93}.

\begin{claim} \label{claim:2chains}
The intersection of two chains is empty or forms a chain.
\end{claim}

\begin{proof}
If the intersection is at most one vertex then the statement clearly holds. Suppose that $A$ and $B$ are chains with at least two common vertices and their intersection is not a chain.
Let $A=v_1, \, v_2, \, \dots, \, v_a$. We can assume without loss of generality that $v_1, \, v_a \in B$ but no other vertex of $A$ is in $B$. Otherwise, we can delete some vertices of $A$ to obtain this situation. Delete all vertices of $B$ that are not between $v_1$ and $v_a$. Now $B=u_1, \, u_2, \, \dots, \, u_b$ where $v_1=u_1$, $v_a=u_b$ and these are the only common points of $A$ and $B$. Let $R$ be the rhombus that represents $v_1=u_1$ in $G_0$. Its sides have slopes $w_1$ and $w_2$ such that $A$ is a $w_1$-chain, $B$ is a $w_2$-chain. Apply an affine transformation so that $R$ is a unit square, $w_1$ is the horizontal, $w_2$ is the vertical direction. Suppose that $Q$ is the rhombus that represents $v_a=u_b$. Then its sides also have slopes $w_1$ and $w_2$, so $Q$ is also an axis parallel unit square. Represent each vertex  $v_1, \, v_2, \, \dots, \, v_a, \,
u_1, \, u_2, \, \dots, \, u_b$ by the center of the corresponding rhombus. For simplicity, we call these points also 
$v_1, \, v_2, \, \dots, \, v_a, \, u_1, \, u_2, \, \dots, \, u_b$, respectively. Assume without loss of generality that 
the point $v_a=u_b$ has larger $x$ and $y$ coordinates than $v_1=u_1$. Connect the consecutive points in both chains by straight line segments. Since $A$ is a $w_1$-chain and $w_1$ is the horizontal direction, the polygonal chain $P_A=(v_1, \, v_2, \, \dots, \, v_a)$ is $y$-monotone, and similarly, the polygonal chain $P_B=(u_1, \, v_2, \, \dots, \, v_b)$ is $x$-monotone.
Let $l_1$ be the horizontal halfline from $v_1=u_1$, pointing to the left and let $l_2$ be the horizontal halfline from $v_a=u_b$, pointing to the right. The bi-infinite curve $l_1\cup P_B\cup l_2$ is simple, because $P_B$ is $x$-monotone. It divides the plane into two regions, $R_{\text{down}}$, which is below it and its complement, $R_{\text{up}}$, see \Cref{chains}.

\begin{figure}[htp]
\centering
\includegraphics[width=11cm]{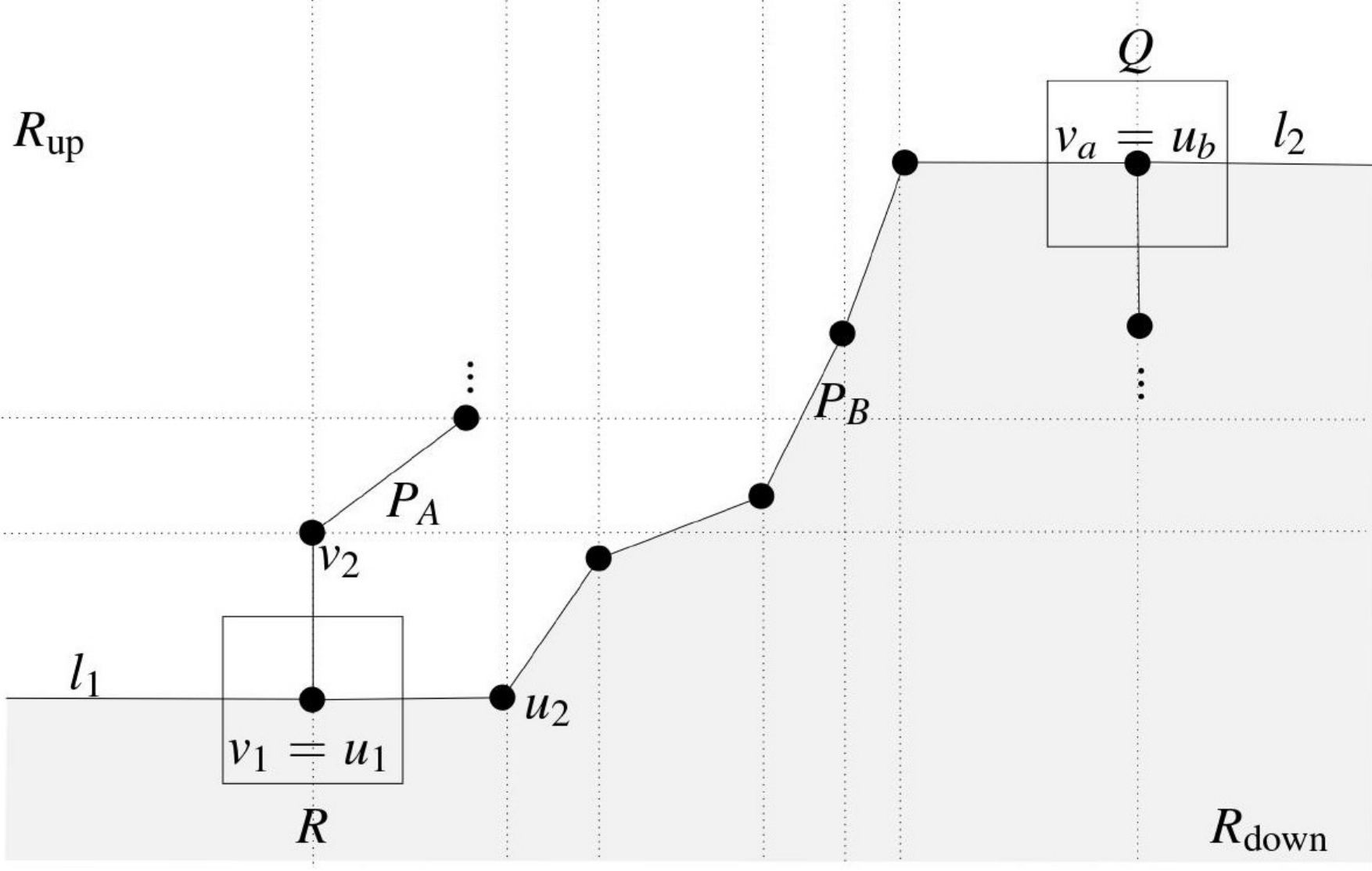}
\caption{The intersection of two chains is empty or forms a chain.}
\label{chains}
\end{figure}

Observe, that the initial part of $P_A$, near $v_1=u_1$ is in $R_{\text{up}}$ while the final part, near $v_a=u_b$ is in  $R_{\text{down}}$. On the other hand, $P_A$ does not intersect the boundary of $R_{\text{down}}$ and $R_{\text{up}}$, Indeed, it does not intersect $l_1$ and $l_2$ since it is $y$-monotone, and does not intersect $P_B$ by assumption. This is clearly a contradiction which proves the Claim.
\end{proof}

\begin{claim} \label{claim:manychains}
There are at least $\sqrt{n}/\sqrt{2}$ different maximal chains.
\end{claim}

\begin{proof}
For any vertex of $H$ (that is, for any rhombus face in $G_0$) there are exactly two maximal chains containing it. Therefore, the total length of all the maximal chains is $2f_4\ge n/2$. If there are less than $\sqrt{n}/\sqrt{2}$ different maximal chains, then one of them, say $C$, has length at least 
$\sqrt{n}/\sqrt{2}$. Through each of its vertices, there is another maximal chain and by \Cref{claim:2chains} all of these chains are different. 
\end{proof}

By \Cref{claim:manychains}, we have at least $\sqrt{n}/\sqrt{2}$ different maximal chains. Each of them has two ending edges, which bounds a face of size different than $4$. All of these bounding edges are different, therefore, $3f_3+F_{\ge 5}\ge \sqrt{2}\sqrt{n}$, which implies that either 
$3f_3 \ge \sqrt{n}/\sqrt{2}$ or 
$F_{\ge 5}\ge \sqrt{n}/\sqrt{2}$.

In the first case, by \Cref {claim:f3} and using that $n\ge 5$, we have $e\le 3n-\sqrt[4]{n}/15$. In the second case, by \Cref{claim:f>=5}, we have $e\le 3n-\sqrt{n}/15$. This concludes the proof of \Cref{1-planar}.
\end{proof}

\begin{remark} We did not attempt to optimize the coefficient of the term 
$\sqrt[4]{n}$ in \Cref{1-planar}. A slightly more careful calculation
gives $u_1(n) \leq 3n-\sqrt[4]{n}/10$ and it can be further improved.
\end{remark}

%%%%%%%%%%%%%%%%%%%%%%%%%%%%%%%%%%%%%%%%%%%%%%%%%%%%%%%%%%%%%%%%%%%%%
\section{{\em k}-planar unit distance graphs} \label{sec3}

\begin{proof}[\bf Proof of \Cref{rote}.]
Suppose that $n, \, k >100$.  The following is a well-known result in number theory, see \cite{PA11, M13}. For any $m$, there is an $r<m$ such that $r$ can be written as $a^2+b^2$ in  $2^{\Omega(\log m/\log\log m)}$ different ways where $a$ and $b$ are integers. For any fixed $m$, let $r$ be the product of the first $l$ primes congruent to $1$ mod $4$, such that $l$ is maximal with the property that $r<m$. This $r$ satisfies the requirements. 

Erd\H os \cite{E46} used this observation  to construct a set of $n$ points that determine $n \cdot 2^{\Omega(\log n/\log\log n)}$ unit distances. Clearly, $r$ is square-free, therefore, whenever $r=a^2+b^2$, we have $(a, \, b)=1$.

By applying the above result for $m=\sqrt{k}/5$, we obtain an $r<\sqrt{k}/5$ that can be written as the sum of two integer squares, $r=a^2+b^2$ in
$2^{\Omega(\log m/\log\log m)}= 2^{\Omega(\log k/\log\log k)}$ different ways. Take a 
$\lfloor\sqrt{n}\rfloor\times\lfloor\sqrt{n}\rfloor$ unit square grid (plus some isolated points far away, to have $n$ vertices) 
and connect two grid points by a straight line segment if they are at distance $\sqrt{r}$. Observe that no edge contains a vertex in its interior, thus it defines a geometric graph $G=(V, \, E)$. Almost every vertex of $G$ has degree $2^{\Omega(\log k/\log\log k)}$, so it has $n \cdot 2^{\Omega(\log k/\log\log k)}$ edges.

Let $uv \in E$ be an edge. Consider all vertices adjacent to an edge that crosses $uv$. All these vertices are at distance at most $\sqrt{r}$ from $uv$, see \cref{fig_thm3}. This region, the possible location of the endpoints of the crossing edges,  has area $(2+\pi)r$, so the number of vertices of $G$ in this region is less than $6r$. 
Each of these vertices have degree at most $4r$, 
since $r$ can be written as  $a^2+b^2$ in at most $r$ different ways.
So $uv$ is crossed by at most $24r^2<k$ edges. Scale the picture by a factor of $1/\sqrt{r}$ and we obtain a $k$-planar unit distance graph of $n$ vertices and $2^{\Omega(\log k/\log\log k)}$ edges.
\end{proof}

\begin{figure}[htp]
\centering
\includegraphics[width=7cm]{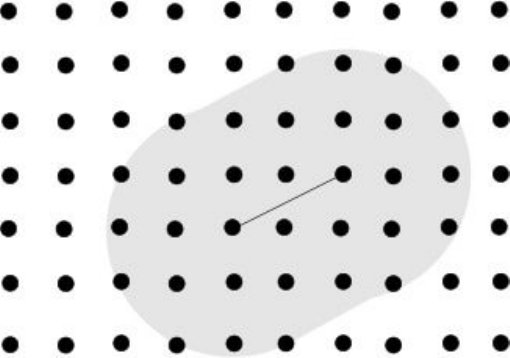}
\smallskip

\caption{For a given edge $uv$, all vertices that are adjacent to an edge that crosses $uv$ are contained in the shaded region.}
\label{fig_thm3}
\end{figure}

\bigskip

%%%%%%%%%%%%%%%%%%%%%%%%%%%%%%%%%%%%%%%

For the proof of \Cref{k-planar}, we  need some preparation. The {\em crossing number} $\cro(G)$ is the minimum number of edge crossings over all drawings of $G$ in the plane. According to the Crossing Lemma \cite{ACNS82, L84}, for every graph $G$ with $n$ vertices and $e\ge 4n$ edges, ${\cro}(G)\ge \frac{1}{64}\frac{e^3}{n^2}$. It is asymptotically tight in general for simple graphs \cite{PT97}. However, there are better bounds for graphs satisfying some monotone property \cite{PST00}, or for graphs drawn in monotone drawing styles \cite{KPTU21}. A drawing style $\cal D$ is a subset of all drawings of a graph $G$, so some drawings belong to $\cal D$, others do not. A drawing style is monotone if removing edges retains the drawing style, that is, for every graph $G$ in drawing style $\cal D$ and any edge removal, the resulting graph with its inherited drawing is again in drawing style $\cal D$.

A vertex split is the following operation: (a) Replace a vertex $v$ of $G$ by two vertices, $v_1$ and $v_2$, both very close to $v$. Connect each edge of $G$ incident to $v$ either to $v_1$ or $v_2$ by locally modifying them  such that no additional crossing is created. Or as an extreme or limiting case, (b) place both $v_1$ and $v_2$ to the same point where $v$ was, connect each edge incident to $v$ either to $v_1$ or $v_2$ without modifying them, such that the edges incident to $v$ in $G$ that are connected to $v_1$ (resp.\ $v_2$) after the split form an interval in the clockwise order from $v$. See \cref{fig_vertexsplit}. A drawing style $\cal D$ is split-compatible if performing vertex splits retains the drawing style.

\begin{figure}[htp]
\centering
\includegraphics[width=11cm]{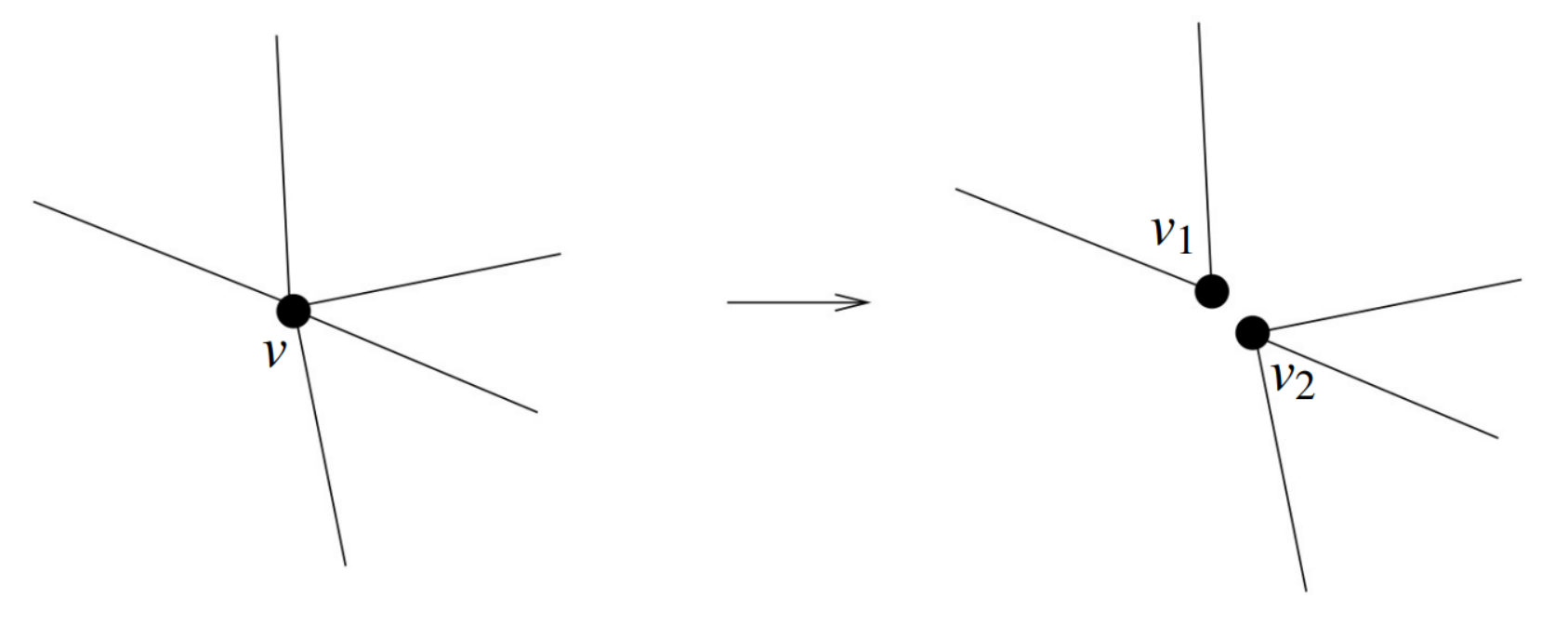}
\caption{A vertex split.}
\label{fig_vertexsplit}
\end{figure}

For any graph $G$, the $\cal D$-crossing number, 
$\cro_{\cal D}(G)$ is  
the minimum number of edge crossings over all drawings of $G$ in the plane, in drawing style $\cal D$.
The bisection width $b(G)$ of $G$ is the smallest number of edges whose removal splits $G$ into two graphs, $G_1$ and $G_2$, such that $|V(G_1)|$, $|V(G_2)| \geq |V(G)|/5$.
For a drawing style $\cal D$ the $\cal D$-bisection width $b_{\cal D}(G)$ of a graph $G$ in drawing style $\cal D$ is the smallest number of edges whose removal splits $G$ into two graphs, $G_1$ and $G_2$, both in drawing style $\cal D$ such that $|V(G_1)|$, $|V(G_2)| \geq |V(G)|/5$. Let $\Delta(G)$ denote the maximum degree in $G$. The following result is a generalization of the Crossing Lemma.

\begin{theorem} [Kaufmann, Pach, T\'oth, Ueckerdt \cite{KPTU21}] \label{Kaufmann-Pach-Toth-Ueckerdt}
Suppose that $\cal D$ is a monotone and split-compatible drawing style, and there are constants $k_1$, $k_2$, $k_3 >0$ and
$b > 1$ such that each of the following holds for every graph $G$ with $n$ vertices and $e$ edges drawn in style $\cal D$:
\begin{itemize}
\item[1.]  If ${\cro}_{\cal D}(G) = 0$, then $e \leq k_1 \cdot n$.
\item[2.] The $\cal D$-bisection width satisfies 
$b_{\cal D}(G) \leq k_2 \sqrt{{\cro}_{\cal D}(G) + \Delta(G) \cdot e + n}$.
\item[3.] $e \leq k_3 \cdot n^b$.
\end{itemize}

Then there exists a 
constant $\alpha > 0$ such that for any graph $G$ with $n$ vertices and $e$ edges, drawn in drawing style $\cal D$, we have

$${\cro}_{\cal D}(G) \geq\alpha\frac{e^{1/(b-1)+2}}{n^{1/(b-1)+1}} \text{ \qquad provided } e > (k_1 + 1)n. $$
\end{theorem}

In \cite{KPTU21} only vertex split of type (a) was allowed, but the proof goes through also for type (b).

\begin{theorem} [Spencer, Szemer\'edi, Trotter \cite{SST84}]
    Let $G$ be a unit distance graph on $n$ vertices. The number of edges in $G$ is at most $c n^{4/3}$ where $c>0$ is a constant.
\end{theorem}

The best known constant is due to \'Agoston and P\'alv\"olgyi \cite{AP22} who proved that the statement holds with $c=1.94$.

\smallskip

\begin{proof}[\bf Proof of \Cref{k-planar}.]
Consider now the following drawing style $\cal D$ for a graph $G$.

\begin{itemize}
\item[1.] Vertices are represented by not necessarily distinct points of the plane. 
\item[2.]  Edges are represented by unit segments between the corresponding points.
\item[3.]  The intersection of two  edges is empty or a point, that is, edges cannot overlap.
\item[4.]  If a point $p$ represents more than one vertex, say $v_1, \, \dots, \, v_m$, then 
the sets of edges incident to $v_1, \, \dots, \, v_m$, respectively, form an interval in the clockwise order
from point $p$. 
\end{itemize}

Clearly,  ${\cal D}$ satisfies the following properties.
\begin{itemize}
\item[1.] The drawing style $\cal D$ is monotone and split-compatible. 
\item[2.]  If ${\cro}(G) = 0$, then $e \leq 3n-6$. In fact, by \cite{LS22}, $e\le \lfloor 3n - \sqrt{12n - 3}\rfloor$. 
\item[3.] For any graph $G$, we have  $b(G)\le 10\sqrt{{\cro}(G) + \Delta(G) \cdot e + n}$ by the result of 
Pach, Shahrokhi and Szegedy \cite{PSS94}.
But if $G$ is drawn in drawing style ${\cal D}$, then all of its subgraphs are also drawn in drawing style  ${\cal D}$. Therefore, $b_{\cal D}(G)\le 10\sqrt{{\cro}(G) + \Delta(G) \cdot e + n}$. 
\item[4.]  By \cite{AP22}, any graph with $n$ vertices drawn in  style $\cal D$ has at most $1.94n^{4/3}$ edges.
\end{itemize}

Summarizing, we can apply \Cref{Kaufmann-Pach-Toth-Ueckerdt} with $k_1=3$, $k_2=10$, $k_3=1.94$, $b=4/3$ and obtain the following. For any graph $G$ in drawing style $\cal D$ with $n$ vertices and $e>4n$ edges, we have 
$$\cro_{\cal D}(G) \geq \alpha \frac{e^{1/(b-1)+2}}{n^{1/(b-1)+1}}=\alpha\frac{e^5}{n^4}$$
for some $\alpha >0$.

Consider now a $k$-plane drawing of a unit distance graph $G$ with $n$ vertices and $e$ edges. If $e\le 4n$, we are done, so suppose that $e\ge 4n$. Since each edge contains at most $k$ crossings, the total number of crossing $c(G)$ satisfies $c(G)\le ek/2$. On the other hand, we have $c(G)\ge \alpha\frac{e^5}{n^4}$.
Therefore, $ek/2\ge \alpha\frac{e^5}{n^4}$ so 
$e\le \beta\sqrt[4]{k}n$ for some $\beta>0$.
\end{proof}

%%%%%%%%%%%%%%%%%%%%%%%%%%%%%%%%%%%%%%%%%%%%%%%%%%%%%%%%%%%%%%%
\section{{\em k}-quasiplanar unit distance graphs}\label{sec4}

\begin{proof}[\bf Proof of \Cref{quasi}.]
Let $G=(V, \, E)$ be a unit distance graph with $|V|=n$ vertices drawn in the plane with no $k$ pairwise crossing edges. We can assume without loss of generality that no pairs of vertices of $G$ determine an angle of $0$, $\pi/2$ or $\pi$ with the $x$-axis. We can also assume that at least half of the edges have positive slope, otherwise, we rotate the coordinate system by $\pi/2$. Let $E_1$ be the set of edges with positive slope. Delete all other edges from $G$ and let $G_1=(V, \, E_1)$ be the obtained graph.

Substitute now each $e\in E_1$ by two directed edges, both drawn as straight line segments, $e^l$, pointing to the left, and $e^r$, pointing to the right, and let $G_2=(V, \, E_2)$ be the resulting directed graph. We have $|E_2|\ge |E|$. Edges that point to the left (resp.\ right) are called {\em left edges} (resp.\ {\em right edges}). Now we decompose the edges into {\em blocks}. Edges in a given block will form a directed path. Suppose that $e$ is a right edge, let $r$ be its right endpoint. Let $f$ be the left edge, whose right endpoint is also $r$ and it is immediately below $e$ (if such an $f$ exists). We say that $f$ is the {\em continuation} of $e$. Suppose now that $e$ is a left edge, let $l$ be its left endpoint. Let $f$ be the right edge, whose left endpoint is also $l$ and it is immediately above $e$ (if such an $f$ exists). In this case we also say that $f$ is the {\em continuation} of $e$.

Define the binary relation between edges in $E_2$ as follows.
For any $e, \, f\in E_2$, $e\sim f$ if $e$ is a continuation of $f$ or $f$ is a continuation of $e$. Finally, let the relation $\approx$ be the transitive closure of $\sim$. In other words, we can take a continuation of an edge arbitrarily many times. The relation $\approx$ is an equivalence relation, its equivalence classes are called {\em blocks}.

\smallskip

\begin{lemma} \mbox{}

\begin{enumerate}[label=(\roman*)] 
    \item Each block is a path of at most $2k-2$ edges.
    \item There are at most $2n$ blocks.
\end{enumerate}
\end{lemma}

\begin{proof}
(i) An edge has 
at most one continuation. Therefore, when we take the transitive closure of $\sim$, we simply iterate the continuation operation. Consequently, every block is a directed walk. Note that the slope of any edge is smaller than the slope of its continuation, therefore, the same edge cannot appear twice on this walk, consequently, every block is a trail.

For any edge $e\in E_2$, let $R(e)$ be the smallest axis-parallel rectangle that contains $e$. Since all edges have positive slopes,
the right endpoint of $e$ is the upper right corner of $R(e)$ and the left endpoint is the lower left endpoint. For any axis-parallel rectangle $R$, let $\left[x_1(R), \, x_2(R)\right]$ (resp.\ $\left[y_1(R), \, y_2(R)\right]$) be its projection on the $x$-axis (resp.\ $y$-axis).

Let $e$ be a right edge. Observe that if $f$ is the continuation of $e$ then
\begin{align} \label{subsets1}
x_1(R(e)) < x_1(R(f)) < x_2(R(f)) = x_2(R(e))
\quad \text{and} \quad 
y_1(R(f)) < y_1(R(e)) < y_2(R(e)) = y_2(R(f)).
\end{align}

Similarly, if $e$ is a left edge and $f$ is its continuation then

\begin{align} \label{subsets2}
x_1(R(e)) = x_1(R(f)) < x_2(R(f)) < x_2(R(e))
\quad \text{and} \quad 
y_1(R(f)) = y_1(R(e)) < y_2(R(e)) < y_2(R(f)).
\end{align}

Suppose now that $(e_1, \, e_2, \, \dots, \, e_m)$ is a trail, which is an equivalence class of $\approx$, and for any
$i$, $1\le i<m$, $e_{i+1}$ is the continuation of $e_i$.
Using (\ref{subsets1}) and (\ref{subsets2}), we obtain that
$x_1(R(e_{i}))<x_1(R(e_{i+2}))<x_2(R(e_{i+2}))<x_2(R(e_{i}))$ (see \Cref{quasiplanar}).

\begin{figure}[ht]
\centering
\includegraphics[width=9cm]{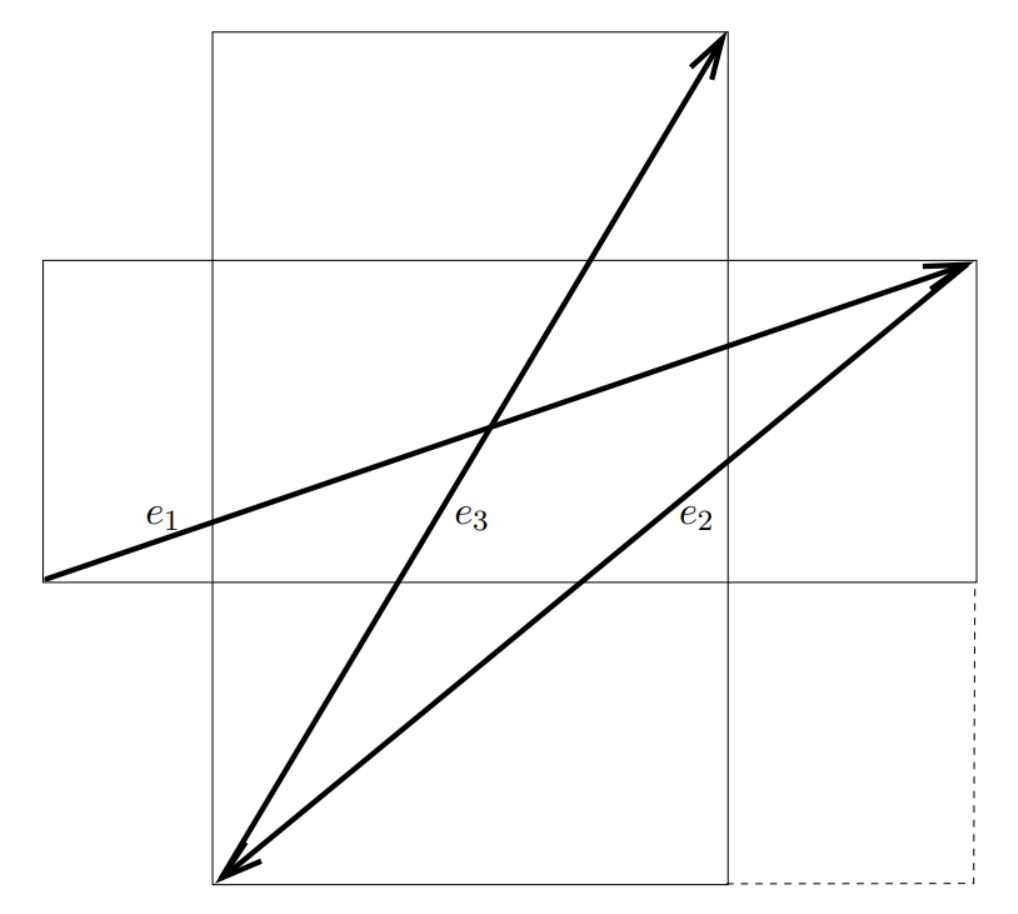}
\caption{$x_1(R(e_1))<x_1(R(e_3))<x_2(R(e_3))<x_2(R(e_1)$ and
$y_1(R(e_3))<y_1(R(e_1))<y_2(R(e_1))<y_2(R(e_3))$.}
\label{quasiplanar}
\end{figure}

It follows that for any $1\le i, \, j\le m$, $j\ge i+2$, we have
$x_1(R(e_i))<x_1(R(e_j))<x_2(R(e_j))<x_2(R(e_i))$. 

By a similar argument, for any $1\le i, \, j\le m$, $j\ge i+2$, we have
$y_1(R(e_{j}))<y_1(R(e_i))<y_2(R(e_i))<y_2(R(e_{j}))$. 

This implies that each equivalence class is a {\em directed path}, since no vertex can be repeated on it. It also implies that $e_1, \, e_3, \, \dots, \, e_{2s+1}$ ($2s+1\le m$) are independent edges and they are pairwise crossing. Since we do not have $k$ pairwise crossing edges in $G$, we have $m\le 2k-2$.

(ii) Suppose that a block, which is a path
$P= e_1, \, e_2, \, \dots, \, e_m$, ends in vertex $v\in V$.
Then $e_m$ has no continuation. If $v$ is the right endpoint of $e_m$, then $e_m$ is the lowest of all edges in $G$ with right endpoint $v$. Similarly, if $v$ is the left vertex of $e_m$, then $e_m$ is the highest of all edges in $G$ with left endpoint $v$. Therefore, in each vertex at most two blocks can end, so there are at most $2n$ blocks.
\end{proof}

Now \Cref{quasi} (i) follows directly. Edges in $E_2$ are divided into blocks. Each edge belongs to exactly one block, and there are at most $2n$ blocks, each of size at most $2k$. Therefore, $|E|\le |E_2|< 4kn$.

\smallskip

Next, we prove \Cref{quasi} (ii). Suppose that  $k\le 2^{c\left(\log n/\log\log n \right)}n$. Consider again the construction of Erd\H os from the proof of \Cref{rote}.
Take a  $\sqrt{n}\times\sqrt{n}$ unit square grid and a square-free number $r=o(n)$ that can be written as
$a^2+b^2$ in $m\ge k$ different ways where $a$ and $b$ are integers. Connect two grid points by a straight line segment
if they are at distance $\sqrt{r}$. Observe that no edge contains a vertex in its interior, thus it defines a geometric graph. All edges have one of $m$ distinct directions and almost all vertices have degree $2m$. Now pick $(k-1)$ directions and only keep the edges in these directions. Then all edges have one of the $(k-1)$ distinct directions and almost all vertices have degree $2(k-1)$. So we have $(k-1)n-o(n)$ edges. Since any two crossing edges have different directions, we do not have $k$ pairwise crossing edges and the Theorem follows.
\end{proof}

Observe that the condition $k=2^{O\left(\log n/\log\log n \right)}n$ seems very hard to relax. In order to replace it with a weaker condition, we have to improve the construction of Erd\H os \cite{E46} for the maximum number of unit distances in a planar point set.

%%%%%%%%%%%%%%%%%%%%%%%%%%%%%%%%%%%%%%%%%%%%%
\section{Open questions} \label{open_questions}

In this paper we proved that a $1$-planar unit distance graph on $n$ vertices can have at most $u_1(n)\le 3n-\sqrt[4]{n}/10$ edges. 

Clearly, for every $n$, $u_1(n)\ge u_0(n)=\lfloor 3n - \sqrt{12n - 3}\rfloor$, and 
we could not rule out the possibility that
$u_0(n)= u_1(n)$ for every $n$. 
As we mentioned in the introduction, very recently,  \v{C}ervenkov\'a \cite{C25}
proved that it is not the case. However, her beautiful construction 
has just $u_0(n)+1$ or $u_0(n)+2$ edges, for 
infinitely many values of $n$. Her work is in progress, she claims to have some further improvements. 

\begin{problem}
Is it true that
    \begin{enumerate}[label=(\roman*)]
        \item $3n-u_1(n)=\Omega(\sqrt{n})$?

\item $u_1(n)-u_0(n)=\Omega(\sqrt{n})$?
\end{enumerate}

\end{problem}

Even more surprisingly, for $k=2$, we do not have a construction with asymptotically more than $3n$ edges. 
The best construction is by  D\'aniel Simon (personal communication, 2023) with roughly $3n-\sqrt{8.3n}$ edges. For $k=3$ there is an easy construction (a piece of a unit triangular grid and its shifted copy by a unit vector) with $3.5n - c\sqrt{n}$ edges for some constant $c>0$.

\begin{problem}
 Determine $u_k(n)$, the maximum number of edges of a k-planar unit distance graph. 
\end{problem}

It is easy to construct an $r$-regular 1-planar unit distance graph (or even a matchstick graph) for $r \leq 3$. There are $4$-regular matchstick graphs, the smallest known example, due to  Harborth, has $52$ vertices \cite{Harborth2}, see also \cite{WD17}. It was shown by Kurz \cite{K11} that any such graph has at least $34$ vertices. Since matchstick graphs are planar, there are no $r$-regular matchstick graphs for $r\ge 6$. Blokhuis \cite{B14} and independently, Kurz and Pinchasi \cite{KP11} proved that there are no $5$-regular matchstick graphs. By the above results, there are  $r$-regular 1-planar unit distance graphs for $r\leq 4$. On the other hand, it follows from \Cref{1-planar} that there are no $r$-regular 1-planar unit distance graphs for $r\ge 6$.

\begin{problem}
Are there any 5-regular 1-planar unit distance graphs?
\end{problem}

%%%%%%%%%%%%%%%%%%%%%%%%%%%%%%%%%%%%%%%%%%%%%%%%%%%%%%%%%%%%%%%%%%%%%
\vspace{5mm}

\noindent
{\bf \large Acknowledgments.}
We are very grateful to the anonymous referees for their suggestions. The authors were supported by ERC Advanced Grant `GeoScape' No.\ 882971 and by the National Research, Development and Innovation Office, NKFIH, K-131529. Panna Geh\'er was also supported by the Lend\"ulet Programme of the Hungarian Academy of Sciences -- grant number LP2021-1/2021.

%%%%%%%%%%%%%%%%%%%%%%%%%%%%%%%%%%%%%%%%%%%%%%%%%%%%%%%%%%%%%%%%%%%%%

\end{document}